\documentclass{elsarticle}
\usepackage[latin1]{inputenc} 
\usepackage[T1]{fontenc} 
\usepackage[english]{babel} 
\usepackage{indentfirst} 
\usepackage{latexsym}
\usepackage{amssymb}
\usepackage{amsmath}
\usepackage{amsthm}
\usepackage{url}

\newtheorem{thm}{Theorem}[section]
\newtheorem{cor}[thm]{Corollary}
\newtheorem{lem}[thm]{Lemma}
\newtheorem{prop}[thm]{Proposition}
\newtheorem{defn}[thm]{Definition}

\newtheorem*{thm*}{Theorem}

\newcommand{\U}{\mathcal{U}}
\newcommand{\N}{\mathbb{N}}
\newcommand{\Z}{\mathbb{Z}}

\newcommand{\V}{\mathcal{V}}

\newcommand{\bN}{\beta\mathbb{N}}

\begin{document}

\title{A nonstandard technique in combinatorial number theory}
\author{Lorenzo Luperi Baglini\corref{cor1}}
\ead{lorenzo.luperi.baglini@univie.ac.at}
\cortext[cor1]{The author has been supported by grant P25311-N25 of the Austrian Science Fund FWF.}
\address{University of Vienna, Faculty of Mathematics, Oskar-Morgenstern-Platz 1, 1090 Vienna, Austria.}

\begin{abstract}

In \cite{rif84}, \cite{Tesi} it has been introduced a technique, based on nonstandard analysis, to study some problems in combinatorial number theory. In this paper we present three applications of this technique: the first one is a new proof of a known result regarding the algebra of $\bN$, namely that the center of the semigroup $(\bN,\oplus)$ is $\N$; the second one is a generalization of a theorem of Bergelson and Hindman on arithmetic progressions of lenght three; the third one regards the partition regular polynomials in $\mathbb{Z}[X]$, namely the polynomials in $\mathbb{Z}[X]$ that have a monochromatic solution for every finite coloration of $\N$. We will study this last application in more detail: we will prove some algebraical properties of the sets of such polynomials and we will present a few examples of nonlinear partition regular polynomials.\\
In the first part of the paper we will recall the main results of the nonstandard technique that we want to use, which is based on a characterization of ultrafilters by means of nonstandard analysis.

\end{abstract}

\maketitle
\section{Introduction}

Many problems in combinatorial number theory are related with the study of "partition regular families", that are defined as follows:

\begin{defn}[] Let $\mathcal{F}$ be a family, closed under superset, of nonempty subsets of a set $S$. $\mathcal{F}$ is {\bfseries partition regular} if, whenever $S=A_{1}\cup...\cup A_{n}$, there exists an index $i\leq n$ such that $A_{i}\in \mathcal{F}$.
\end{defn}

The partition regular families on $S$ are closely related to the ultrafilters on $S$:

\begin{thm}\label{ultra} Let $S$ be a set and $\mathcal{F}$ a family, closed under supersets, of nonempty subsets of $S$. Then $\mathcal{F}$ is partition regular if and only if there exists an ultrafilter $\U$ on $S$ such that $\U\subseteq\mathcal{F}$.
\end{thm}
\begin{proof} This is Theorem 3.11 in \cite{rif12}. \end{proof}

So ultrafilters are an important tool to study many problems in combinatorial number theory. In some recent works it has been introduced and used a technique to study ultrafilters on the set of natural numbers\footnote{In this paper, to simplify some definitions and some results, we assume that $\N=\{1,2,3,...\}$.} by means of Nonstandard Analysis. The basic idea of this technique is that, given a nonstandard extension $^{*}\N$ of $\N$ (with a particular technical condition that will be introduced later), every ultrafilter $\U$ can be identified with a subset $G_{\U}$ of $^{*}\N$. Following an approach that has some ideas in common with the one used by Christian W. Puritz in his articles \cite{rif15}, \cite{rif16}, the one used by Joram Hirschfeld in \cite{rif14} and the one used by Greg Cherlin and Joram Hirschfeld in \cite{rif99}, it can be shown that many combinatorial properties of $\U$ can be deduced by algebraical properties of $G_{\U}$ which, in some case, are easier to prove. E.g., in \cite{rif84} this technique has been used to prove a "qualitative" property of Rado's Theorem, while in \cite{Mono} it has been used to study some results in combinatorial number theory that regard nonlinear polynomials in $\mathbb{Z}[X]$.\\
In this paper we present three new applications of this technique.\\
The first application regards the algebra of $(\bN,\oplus)$, namely the Stone-\v{C}ech compactification of the semigroup $(\N,+)$: we will give a direct proof that the center of the semigroup $(\bN,\oplus)$ is $\N$.\\
The second applications regards arithmetic progressions of lenght three. In \cite{rif41} Bergelson and Hindman proved a "qualitative" result on such arithmetic progressions, namely that, for every idempotent ultrafilter $\U$ and for every set $A\in 2\U\oplus\U$, $A$ contains arithmetic progressions of lenght three. We will reprove this fact, and we will prove that a similar "qualitative" property holds for the ultrafilter $n_{1}\U\oplus... n_{k}\U$, where $k,n_{1},...,n_{k}$ are any natural numbers with $n_{i+1}\neq n_{i}$ for every $i=1,...,k-1$.\\
The third application regards the partition regular polynomials in $\mathbb{Z}[X]$:

\begin{defn} A polynomial $P(x_{1},...,x_{n})\in\mathbb{Z}[X]$ is {\bfseries partition regular} (on $\N$) if, for every finite coloration of $\N$, the equation $P(x_{1},...,x_{n})=0$ has a monochromatic solution.\end{defn}

The linear polynomials have been studied by Richard Rado in \cite{rif17}:

\begin{thm}[Rado]\label{Rado}Let $P(x_{1},...,x_{n})= \sum_{i=1}^{n}a_{i}x_{i}$ be a linear polynomial in $\mathbb{Z}[X]$ with nonzero coefficients. The following conditions are equivalent:
\begin{enumerate}
	\item $P(x_{1},...,x_{n})$ is partition regular on $\N$;
	\item there is a nonempy subset $J$ of $\{1,...,n\}$ such that $\sum\limits_{j\in J}a_{j}=0$.
\end{enumerate}
\end{thm}

Very little is known in the case of nonlinear polynomials. We will study a few properties of the partition regular (linear and nonlinear) polynomials, and we will present a particular case of a result proved in \cite{Mono}, namely that the nonlinear polynomial

\begin{center} $P(x,y,z,w)=x+y-zw$ \end{center}

is partition regular.

\section{The Nonstandard Approach}

In this section we recall the basic ideas and the basic facts regarding the nonstandard technique that we want to use in the following. A more detailed presentation can be found in \cite{rif84}, \cite{Tesi} and \cite{Mono}.\\
We assume the knowledge of the nonstandard notions and tools that we use, in particular the knowledge of superstructures, star map and enlarging properties (see, e.g., \cite{rif5}), as well as the knowledge of the basic facts and definitions regarding the Stone-\v{C}ech compactification $\beta S$ of a semigroup $(S,+)$ (we notice that we will be interested only in the cases $(\N,+)$, $(\N,\cdot)$ and $(\mathbb{Q},+)$). We suggest \cite{rif12} as a general reference about ultrafilters, \cite{rif1} and \cite{rif19} as introductions to nonstandard methods and \cite{rif5} as a reference for the model theoretic notions that we use.\\
The first notion that we recall is the following:
\begin{defn} A {\bfseries superstructure model of nonstandard methods} is a triple $\langle \mathbb{V}(X), \mathbb{V}(Y), *\rangle$ where 
\begin{enumerate}
	\item a copy of $\N$ is included in $X$ and in $Y$;
	\item $\mathbb{V}(X)$ and $\mathbb{V}(Y)$ are superstructures on the infinite sets $X$, $Y$ respectively;
	\item $*$ is a proper star map from $\mathbb{V}(X)$ to $\mathbb{V}(Y)$ that satisfies the transfer property.
\end{enumerate}
\end{defn}

We will be interested in single superstructure models of nonstandard methods, i.e. models where $\mathbb{V}(X)=\mathbb{V}(Y)$. These models (which actually exists, see \cite{rif2}) have a nice property that will be of great importance in the following: they allow to iterate the star map. In particular, for any object $\alpha\in\mathbb{V}(X)$, the element $^{*}\alpha$ is well defined and, due to the transfer property, it has "a lot of properties in common with $\alpha$" even if, in general, $\alpha$ and $^{*}\alpha$ will be different. For example, $^{**}\N$ is a well defined object in $\mathbb{V}(X)$, and it is a proper nonstandard extension of $^{*}\N$ (in the same way in which $^{*}\N$ is a nonstandard extension of $\N$). More interestingly, if $\alpha\in$$^{*}\N\setminus\N$, then $^{*}\alpha$ is a well defined object in $^{**}\N\setminus$$^{*}\N$. In this case, $\alpha<$$^{*}\alpha$ (note that this property does not hold if $\alpha=n\in\N$ since, as usual in Nonstandard Analysis, we assume that $^{*}n=n$ for every $n\in\N$). Nevertheless, $\alpha$ and $^{*}\alpha$ shares a lot of properties: e.g., $\alpha$ is prime/odd/a square if and only if $^{*}\alpha$ is prime/odd/a square.\\
The starting point for the construction of our technique is a result that associates elements of $\bN$ and subsets of any "enlarged enough" nonstandard extension of $\N$, which can be obtained by mean of the following theorem:

\begin{thm} $(1)$ Let $^{*}\N$ be a hyperextension of $\N$. For every hypernatural number $\alpha$ in $^{*}\N$, the set 
\begin{center} $\mathfrak{U}_{\alpha}=\{A\in\N\mid \alpha\in$$^{*}A\}$ \end{center} is an ultrafilter on $\N$.\\
$(2)$ Let $^{*}\N$ be a hyperextension of $\N$ with the $\mathfrak{c}^{+}$-enlarging property\footnote{We recall that a nonstandard extension $^{*}S$ of a set $S$ has the $\mathfrak{c}^{+}$-enlarging property if, for every family $\mathfrak{F}$ of subsets of $S$ with $|\mathfrak{F}|\leq \mathfrak{c}$, we have $\bigcap_{A\in\mathfrak{F}}$$^{*}A\neq\emptyset$.}. For every ultrafilter $\U$ on $\N$ there exists an element $\alpha$ in $^{*}\N$ such that $\U=\mathfrak{U}_{\alpha}$.\end{thm}

The previous Theorem is proved, e.g., in \cite{rif97}. Therefore, it is natural to introduce the following definition:

\begin{defn} Given an ultrafilter $\U$ on $\N$ and a $\mathfrak{c}^{+}$-enlarged hyperextension $^{*}\N$ of $\N$, the {\bfseries set of generators}\footnote{Usually, the set $G_{\U}$ is called "monad of $\U$" (see e.g. \cite{rif98}); here we prefer to call these elements "generators" of $\U$ because, as we will show later, many properties of $\U$ are actually generated by properties of the elements in $G_{\U}$.} {\bfseries of $\U$} is
\begin{center} $G_{\U}=\{\alpha\in$$^{*}\N\mid \U=\mathfrak{U}_{\alpha}\}$. \end{center}
\end{defn}

The idea is that many combinatorial properties of $\U$ can be deduced by properties of the elements in $G_{\U}$. The exact formulation of this idea is given in the following Theorem: 

\begin{thm}[Bridge Theorem] Let $\phi(x_{1},...,x_{n})$ be a first order formula in the first order theory of arithmetic, and let $x_{1},...,x_{n}$ be the only free variables of $\phi(x_{1},...,x_{n})$. Let $\U$ be an ultrafilter in $\bN$. The following conditions are equivalent:
\begin{enumerate}
	\item $\forall A\in\U$ there are $a_{1},...,a_{n}\in A$ such that $\phi(a_{1},...,a_{n})$ holds;
	\item there are elements $\alpha_{1},...,\alpha_{n}$ in $G_{\U}$ such that $^{*}\phi(\alpha_{1},...,\alpha_{n})$ holds.
\end{enumerate}
\end{thm}

\begin{proof} This Theorem is proved in \cite{Tesi}.\end{proof}

In this paper we will talk about ultrafilters on $\N$ and $\mathbb{Q}$; it is not difficult to prove that every ultrafilter $\U$ on $\N$ can be extended (by closing under superset) to an ultrafilter on $\mathbb{Q}$ (which we will still denote by $\U$); conversely, to an ultrafilter $\U$ on $\mathbb{Q}$ it can be associated an ultrafilter on $\N$ by restriction if and only if $\N\in\U$ and, in this case, we will identify the ultrafilter $\U$ with its restriction. In terms of generators, if $\U\in\beta\mathbb{Q}$ and $\alpha\in G_{\U}$, then $\U$ can be identified with an ultrafilter on $\N$ if and only if $\alpha\in$$^{*}\N$. Moreover in $\beta\mathbb{Q}$ we can define the operator of "reciprocal of an ultrafilter $\U$": namely, given an ultrafilter $\U\in\beta\mathbb{Q}\setminus\mathfrak{U}_{0}$, we will denote by $\frac{1}{\U}$ the ultrafilter

\begin{center} $\frac{1}{\U}=\{A\subseteq\mathbb{Q}\mid \{a^{-1}\mid a\in A\}\in\U\}$. \end{center}

It is easy to prove that $G_{\frac{1}{\U}}=\{\alpha^{-1}\mid \alpha\in G_{\U}\}$.\\
Finally, we recall that an ultrafilter $\U\in\bN$ is called idempotent if $\U\oplus\U=\U$, where $\oplus$ is the unique right continuous extension of the sum $+:\N^{2}\rightarrow\N$ to $\bN$. Likewise, the ultrafilter $\U$ is a multiplicative idempotent if $\U\odot\U=\U$, where $\odot$ is is the unique right continuous extension of the product $\cdot:\N^{2}\rightarrow\N$ to $\bN$.\\
Many applications of ultrafilters in combinatorial number theory are based on idempotent ultrafilters. So, if we want to apply our nonstandard point of view, it is important to be able to express the operations $\oplus,\odot$ in terms of the sets of generators. This can be obtained by iterating the star map, as we are going to show:

\begin{defn} For every natural number $n$ we define the function 
\begin{equation*} S_{n}:\mathbb{V}(X)\rightarrow\mathbb{V}(X)\end{equation*}
by setting
\begin{equation*}S_{1}=*\end{equation*}
and, for $n\geq 1$, 
\begin{equation*}S_{n+1}=*\circ S_{n}.\end{equation*}
\end{defn}

\begin{defn} Let $\langle\mathbb{V}(X),\mathbb{V}(X),*\rangle$ be a single superstructure model of nonstandard methods. We call {\bfseries $\omega$-hyperextension} of $\N$, and we denote by $^{\bullet}\N$, the union of all the hyperextensions $S_{n}(\N)$:
\begin{center} $^{\bullet}\N=\bigcup\limits_{n\in\mathbb{N}} S_{n}(\N)$. \end{center}
\end{defn}

We observe that, as a consequence of the Elementary Chain Theorem, $^{\bullet}\N$ is a nonstandard extension of $\N$. Moreover, it is easy to prove that if $^{*}\N$ has the $\mathfrak{c}^{+}$-enlarging property then also $^{\bullet}\N$ has the $\mathfrak{c}^{+}$-enlarging property. We will always assume that the $\mathfrak{c}^{+}$-enlarging property holds.\\
To the elements of $^{\bullet}\N$ is associated a notion of "height":

\begin{defn} Let $\alpha\in$$^{\bullet}\N\setminus\N$. The {\bfseries height} of $\alpha$ $($denoted by $h(\alpha))$ is the least natural number $n$ such that $\alpha\in S_{n}(\N)$. \end{defn}

The height of an hypernatural number is needed to translate the operations $\oplus,\odot$ in terms of generators:

\begin{prop}\label{stanfa} Let $\alpha,\beta\in$$^{\bullet}\N$, $\U=\mathfrak{U}_{\alpha}$ and $\V=\mathfrak{U}_{\beta}$. Then:
\begin{enumerate}
	\item for every natural number $n$, $\mathfrak{U}_{\alpha}=\mathfrak{U}_{S_{n}(\alpha)}$;
	\item $\alpha+S_{h(\alpha)}(\beta)\in G_{\U\oplus\V}$;
	\item $\alpha\cdot S_{h(\alpha)}(\beta)\in\ G_{\U\odot\V}.$
\end{enumerate}
\end{prop}

\begin{proof} These results have been proved in \cite{rif84} and in \cite{Tesi}, Chapter 2.\end{proof}

We conclude this section by pointing out that, as a corollary of Proposition \ref{stanfa}, we can easily characterize the property of "being idempotent" in terms of generators:

\begin{prop}\label{idultragen} Let $\U\in\bN$. Then:
\begin{enumerate}
	\item $\U\oplus\U=\U\Leftrightarrow\forall\alpha,\beta\in G_{\U}$ $\alpha+S_{h(\alpha)}(\beta)\in G_{\U}\Leftrightarrow\exists\alpha,\beta\in G_{\U} \alpha+S_{h(\alpha)}(\beta)\in G_{\U}$;
	\item $\U\odot\U=\U\Leftrightarrow\forall\alpha,\beta\in G_{\U}$ $\alpha+S_{h(\alpha)}(\beta)\in G_{\U}\Leftrightarrow\exists\alpha,\beta\in G_{\U}$ $\alpha\cdot S_{h(\alpha)}(\beta)\in G_{\U}$.
\end{enumerate}
\end{prop}
\begin{proof} The thesis follows easily by points (2) and (3) of Proposition \ref{stanfa}. \end{proof}

\section{Applications}

This section is dedicated to present three different applications of the technique presented in section 2.

\subsection{The Center of $\bN$}

The nonstandard technique presented in section 2 can be used to study some algebraical properties of $\bN$. As an example, we give a nonstandard proof of the following known result (see, e.g., \cite{rif12}):

\begin{prop} The center of $(\bN,\oplus)$ is $\N$. \end{prop}

\begin{proof} It is easy to prove that $\N$ is included in the center of $(\bN,\oplus)$: in fact, if $\U\in\bN$, $\V=\mathfrak{U}_{n}$ is the principal ultrafilter on $n$ and $\alpha$ is a generator of $\U$, then by Proposition \ref{stanfa} we have that $\alpha+n\in G_{\U\oplus\V}$ and $n+\alpha\in G_{\V\oplus\U}$. But $\alpha+n=n+\alpha$, so $\U\oplus\V=\V\oplus\U$.\\
To prove that the center is exactly $\N$ it is then enough to show that, for every non principal ultrafilter $\U$, there exists an ultrafilter $\V$ and a set $A\subseteq\N$ such that $A\in \U\oplus\V\Leftrightarrow A^{c}\in\V\oplus\U$. To prove this fact we let $(a_{k})_{k\in \N}$ be an increasing sequence of natural numbers such that $\lim\limits_{k\rightarrow\infty} (a_{k+1}-a_{k})= +\infty$, and we consider the set 

\begin{center} $A=\bigcup_{k\in \N} [a_{2k},a_{2k+1})$. \end{center}

We observe that, by construction, for every natural number $n$ the set $A$ contains many intervals of length greater than $n$, and the same holds for $A^{c}$. By transfer we deduce that, for every hypernatural number $\mu$, the hyperextension $^{*}A$ contains many intervals of length greater than $\mu$, and the same holds for $^{*}A^{c}$. We consider $^{*}A$, and we let $\alpha\in$$^{*}\N$ be a generator of $\U$. We suppose that $\alpha\in$$^{*}A$ (the proof in the case $\alpha\in$$^{*}A^{c}$ is similar).\\
There are two possibilities:\\
Case 1: For every natural number $n$, $\alpha+n\in$$^{*}A$. Then, by transfer, we have that, for every hypernatural number $\mu\in$$^{*}\N$, $^{*}\alpha+\mu\in$$^{**}A$, so $A\in \mathfrak{U}_{\mu}\oplus\mathfrak{U}_{\alpha}$ for every $\mu\in$$^{*}\N$. If there is an hypernatural number $\mu$ in $^{*}\N$ with $\alpha+$$^{*}\mu\in$$^{**}A^{c}$ (i.e. $A^{c}\in \mathfrak{U}_{\alpha}\oplus\mathfrak{U}_{\mu}$) we can conclude by choosing $\V=\mathfrak{U}_{\mu}$.\\
We know that $^{*}A^{c}$ contains arbitrarily long intervals, so there exists an hypernatural number $\eta$ such that the interval 

\begin{center} $I=[a_{\eta},a_{\eta+1})$\end{center}

has length greater than $\alpha$ and it is included in $^{*}A^{c}$. In particular, by letting $\mu=a_{\eta}$, we have that $\mu+n\in I$ for every natural number $n$ and so, by transfer, for every hypernatural number $\xi\in$$^{*}\N$ we have that $^{*}\mu+\xi\in$$^{*}I\subseteq$$^{**}A^{c}$. Setting $\xi=\alpha$ we get the thesis.\\
Case 2: There exists a natural number $n\in\N$ such that $\alpha+n\in$$^{*}A^{c}$. Then, since the intervals $[a_{2\eta},a_{2\eta+1})$ are infinite for $\eta\in$$^{*}\N\setminus\N$, we have that for every natural number $m\geq n$, $\alpha+m\in$$^{*}A^{c}$. By transfer it follows that, for every hypernatural number $\mu$ in $^{*}\N$ with $\mu\geq n$, $^{*}\alpha+\mu\in$$^{**}A^{c}$; in particular, $A^{c}\in \mathfrak{U}_{\mu}\oplus\mathfrak{U}_{\alpha}$ for every $\mu\in$$^{*}\N\setminus\N$. If there is an hypernatural number $\mu$ in $^{*}\N\setminus\N$ such that $\alpha+$$^{*}\mu\in$$^{**}A$ we can conclude.\\
The way in which such a $\mu$ can be found follows the same ideas of the second part of the first case: this time we find $\eta$ such that $I=[a_{2\eta},a_{2\eta+1})$ is included in $^{*}A$ and has length greater than $\alpha$. So, again, if $\mu=a_{2\eta}$ then by transfer we get that $\alpha+$$^{*}\mu\in$$^{**}I\subseteq$$^{**}A$, and this entails the thesis.\end{proof}

\subsection{A generalization of a Theorem of Bergelson and Hindman}

In this section we want to reprove and generalize a result about arithmetic progressions proved by Bergelson and Hindman in \cite{rif41}. The result we are talking about is the following:

\begin{thm}\label{Bergelson} If $\U$ is an idempotent ultrafilter then every set $A\in 2\U\oplus\U$ contains an arithmetic progression of lenght three, namely there are $a<b<c\in A$ with $c-b=b-a$.\end{thm}

We just recall that an immediate Corollary of Theorem \ref{Bergelson} is the following:

\begin{cor}\label{pa3} For every finite coloration of $\N$ there is a monochromatic arithmetic progression of lenght three. \end{cor}

Nevertheless, Theorem \ref{Bergelson} is interesting because it adds a "qualitative" property to \ref{pa3}, namely the fact that the monochromatic set containing the arithmetic progression can always be chosen in $2\U\oplus\U$ for any given idempotent ultrafilter $\U$.\\
We now prove Theorem \ref{Bergelson} with our nonstandard technique:

\begin{proof} Let us consider $\V=2\U\oplus\U$. By applying the Bridge Theorem, we know that to prove the thesis it is enough to find $\alpha,\beta,\gamma\in G_{\V}$ such that $\alpha<\beta<\gamma$ and $\beta-\alpha=\gamma-\beta$.\\
We observe that, since $\U$ is idempotent, then also $2\U$ is idempotent, because $2\U\oplus 2\U=2(\U\oplus\U)=2\U$. Now we let $\eta\in$$^{*}\N$ be any element in $G_{\U}$; then by Proposition \ref{stanfa} and Proposition \ref{idultragen} we have that 
\begin{enumerate}
	\item $\alpha=2\eta+$$^{**}\eta\in G_{\V};$
	\item $\beta=2\eta+$$^{*}\eta+$$^{**}\eta\in G_{\V}$;
	\item $\gamma=2\eta+2$$^{*}\eta+$$^{**}\eta\in G_{\V}.$
\end{enumerate}

$\alpha<\beta<\gamma$ are three generators of $\V$ and they form an arithmetic progression of lenght three, so we have the thesis.\end{proof}

The previous Theorem can be generalized as follows:

\begin{thm}\label{general} Let $\U$ be an idempotent ultrafilter, let $k\in\N$ and let $n_{1},n_{2},...,n_{k+1}$ be natural numbers with $n_{i}\neq n_{i+1}$ for every index $i\leq k$. Then every set \begin{center}$A\in n_{1}\U\oplus n_{2}\U\oplus...\oplus n_{k}\U$\end{center} satisfies the following property: there exist $x_{1},...,x_{k},y_{1},...,y_{k},z_{1},...,z_{k}\in A$ such that the following three conditions hold:
\begin{enumerate}
  \item $\forall i\leq k$, $x_{i}<y_{i}, x_{i}<z_{i};$ 
	\item $\forall i\leq k$, $n_{i}(z_{i}-x_{i})=n_{i+1}(y_{i}-x_{i})$;
	\item $\forall i\leq k-1$, $x_{i+1}=z_{i}$.
\end{enumerate}

\end{thm}

Let us observe that Theorem \ref{Bergelson} is the particular case of Theorem \ref{general} obtained by considering $k=1$ and $n_{1}=2, n_{2}=1$, and that Theorem \ref{general} is clearly much more general: e.g., if $k=2$, $n_{1}=2$, $n_{2}=1$ and $n_{3}=3$, by Theorem \ref{general} it follows that, for every idempotent ultrafilter $\U$, for every set $A\in 2\U\oplus\U\oplus 3\U$ there are $a_{1},b_{1},c_{1},a_{2},b_{2},c_{2}\in A$ such that:
\begin{enumerate}
	\item $a_{1}<b_{1}, a_{1}<c_{1}, c_{1}=a_{2}, a_{2}<b_{2}, a_{2}<c_{2}$;
	\item $2(c_{1}-a_{1})=b_{1}-a_{1}$ and $c_{2}-a_{2}=3(b_{2}-a_{2})$.
\end{enumerate}

\begin{proof} As a consequence of the Bridge Theorem, to prove the claim it is sufficient to show that, given any idempotent ultrafilter $\U$, there are elements $\alpha_{1},...,\alpha_{k},\beta_{1},...,\beta_{k},\gamma_{1},...,\gamma_{k}$ in $G_{\V}$ such that, for every index $i\leq k$, $\alpha_{i}<\beta_{i}<\gamma_{i}$, $n_{i}(\gamma_{i}-\alpha_{i})=n_{i+1}(\beta_{i}-\alpha_{i})$ and $\alpha_{i+1}=\gamma_{i}$, where

\begin{center} $\V=n_{1}\U\oplus...\oplus n_{k+1}\U$.\end{center}

Let $\xi\in G_{\U}$. We construct the elements $\alpha_{i},\beta_{i},\gamma_{i}$ inductively: let

\begin{itemize}
	\item $\alpha_{1}=\sum_{i=1}^{k}(n_{i}S_{2(i-1)}(\xi))$;
	\item $\beta_{1}=\alpha_{1}+n_{1}$$^{*}\xi$;
	\item $\gamma_{1}=\alpha_{1}+n_{2}$$^{*}\xi$.
\end{itemize}

We observe that, by construction, $n_{2}(\beta_{1}-\alpha_{1})=n_{2}\cdot n_{1}$$^{*}\xi=n_{1}(\gamma_{1}-\alpha_{1})$ and $\alpha_{1},\beta_{1},\gamma_{1}$ are generators of $\V$.\\
Now, if $\alpha_{h},\beta_{h},\gamma_{h}$ have been constructed, pose

\begin{itemize}
	\item $\alpha_{h+1}=\gamma_{h}$;
	\item $\beta_{h+1}=\alpha_{h+1}+n_{i+1}S_{2h-1}(\xi)$;
	\item $\gamma_{h+1}=\alpha_{h+1}+n_{i}S_{2h-1}(\xi)$.
\end{itemize}

We observe that $\alpha_{h+1}=\gamma_{h}$, $n_{h+2}(\beta_{h+1}-\alpha_{h+1})=n_{h+1}\cdot n_{h+1}S_{2h-1}(\xi)=n_{h+1}(\gamma_{h+1}-\alpha_{h+1})$ and that $\alpha_{h+1},\beta_{h+1},\gamma_{h+1}$ are generators of $\V$.\\
With this procedure we constuct elements $\alpha_{1},...,\alpha_{k},\beta_{1},...\beta_{k},\gamma_{1},...,\gamma_{k}$ in $G_{\V}$ with the desired properties, so we have the thesis.\end{proof}

\subsection{Partition regular polynomials}

In this section we will prove a few properties of the following two sets:

\begin{defn} The {\bfseries set of partition regular polynomials on $\N$} is 

\begin{center} $\mathcal{P}=\{P\in\Z[\mathbf{X}]\mid P$ is partition regular on $\N$$\}$,\end{center}

and the {\bfseries set of homogeneous partition regular polynomials on $\N$} is

\begin{center} $\mathcal{H}=\{P\in \mathcal{P}\mid P$ is homogeneous$\}$. \end{center}
\end{defn}

While Theorem \ref{Rado} settles the linear case, very little is known in the nonlinear case. We will always assume that the polynomials are given in normal reduced form, that the variables of $P(x_{1},...,x_{n})$ are exactly $x_{1},...,x_{n}$ and that every considered polynomial has costant term 0\footnote{As Rado proved in \cite{rif18}, already in the linear case, when the constant term of the polynomial is not zero the problem of the partition regularity of the polynomial becomes trivial (see also \cite{Mono} for a discussion on this fact).}.\\
Let us observe that the multiples of a polynomial $P(x_{1},...,x_{n})$ in $\mathcal{P}$ are in $\mathcal{P}$ so, in particular, $\mathcal{P}$ is a sub-semigroup of $(\Z[\mathbf{X}], \cdot)$ and $\mathcal{H}$ is a sub-semigroup of $(H[\mathbf{X}],\cdot)$, where $H[\mathbf{X}]=\{P(x_{1},...,x_{n})\in [\Z\mathbf{X}]\mid P(x_{1},...,x_{n})$ is homogeneous\}.\\
As a particular case of Theorem \ref{ultra} it can be proved that a polynomial is partition regular if and only if there exists an ultrafilter $\U$ such that every set $A$ in $\U$ contains a solution to the equation $P(x_{1},...,x_{n})=0$ (see e.g. \cite{Tesi}, \cite{Mono}). So it makes sense to introduce the following definition:

\begin{defn} Let $P(x_{1},...,x_{n})$ be a polynomial and $\U$ an ultrafilter on $\N$. We say that $\U$ is a $\mathbf{\sigma_{P}}${\bfseries-ultrafilter} if and only if for every set $A\in\U$ there are $a_{1},...,a_{n}\in A$ such that $P(a_{1},..,a_{n})=0$.
\end{defn}

The property of being a $\sigma_{P}-$ultrafilter can be studied by means of the following particular case of the Bridge Theorem:

\begin{thm}[Polynomial Bridge Theorem]\label{PBT} Let $P(x_{1},...,x_{n})$ be a polynomial, and $\U$ an ultrafilter on $\bN$. The following two conditions are equivalent:
\begin{enumerate}
	\item $\U$ is a $\sigma_{P}$-ultrafilter;
	\item there are elements $\alpha_{1},...,\alpha_{n}$ in $G_{\U}$ such that $P(\alpha_{1},...,\alpha_{n})=0$.
\end{enumerate}
\end{thm}

As a first example of application of Theorem \ref{PBT} we prove that the study of the partition regularity of polynomials can be restricted to irreducible polynomials. For sake of simplicity, both in the enunciate and in the proof of the following Theorem, when writing $Q_{i}(x_{1},...,x_{n})$ we mean that the set of variables of $Q_{i}$ is a subset of $\{x_{1},...,x_{n}\}$ and, when $\alpha_{1},...,\alpha_{n}$ are hypernatural numbers in $^{\bullet}\N$, by $Q(\alpha_{1},...,\alpha_{n})$ we denote the number obtained by replacing each variable $x_{i}$ in $Q$ by $\alpha_{i}$.

\begin{thm}\label{factoriz} Let $P(x_{1},...,x_{n})\in\Z[X]$, and let us suppose that $P(x_{1},...,x_{n})$ can be factorized as follows: 

\begin{center} $P(x_{1},....,x_{n})=\prod_{i=1}^{k}Q_{i}(x_{1},...,x_{n})$.\end{center} 

If $P(x_{1},...,x_{n})\in \mathcal{P}$ then there exists an index $i$ such that $Q_{i}(x_{1},...,x_{n})\in \mathcal{P}$. \end{thm}

\begin{proof} Let $\U$ be a $\sigma_{P}$-ultrafilter and let $\alpha_{1},...,\alpha_{n}$ be generators of $\U$ such that $P(\alpha_{1},...,\alpha_{n})=0$. Then, for at least one index $i$, $Q_{i}(\alpha_{1},...,\alpha_{n})=0$, so $\U$ is a $\sigma_{Q_{i}}$-ultrafilter, and hence $Q_{i}(x_{1},...,x_{n})$ is in $\mathcal{P}$.\end{proof}

A consequence of Theorem \ref{factoriz} is that, to study the partition regularity of polynomials, it is sufficient to work with irreducible polynomials. In particular, $\mathcal{P}$ is an union of principal ideals in $(\mathbb{Z}[X],\cdot)$ generated by irreducible polynomials. Moreover let us note that, in the same hypothesis of Theorem \ref{factoriz}, the following holds:

\begin{center} $\{\U\in\bN\mid\U$ is a $\sigma_{P}-$ultafilter\}=$\bigcup\limits_{i=1}^{k}\{\U\in\bN\mid\U$ is a $\sigma_{Q_{i}}-$ultafilter\}.\end{center}

Another question that we can ask about $\mathcal{P}$ is: is $\mathcal{P}$ closed under sum? This is clearly false: e.g., the polynomials $P(x,y,z)=x-y+z$ and $Q(x,y,w)=y-x+w$ are partition regular by Rado's Theorem, but their sum is $z+w$ which is not partition regular. As one can imagine, the problem here is that the polynomials $P$ and $Q$ have some variables in common. So we can modify our question as follows: if $P$ and $Q$ are partition regular and do not have variables in common, is $P+Q$ partition regular? In this case, we can prove the following:

\begin{thm}\label{summ} Let $n,m$ be natural numbers and let $P(x_{1},...,x_{n})$, $Q(y_{1},...,y_{m})\in\mathcal{H}$. Moreover, let us suppose that $\{x_{1},...,x_{n}\}\cap\{y_{1},..,y_{m}\}=\emptyset$.\\
Then $P(x_{1},...,x_{n})+Q(y_{1},...,y_{m})\in\mathcal{P}$.\end{thm}

\begin{proof} Let $\U$ be a $\sigma_{P}$-ultrafilter and let $\V$ be a $\sigma_{Q}$-ultrafilter. To prove the result it is sufficient to find a $\sigma_{P+Q}$-ultrafilter. We claim that the ultrafilter $\U\odot\V$ is a $\sigma_{P+Q}$-ultrafilter.\\
To prove our claim, we let $\alpha_{1},...,\alpha_{n}\in$$^{*}\N\cap G_{\U}$, $\beta_{1},...,\beta_{m}\in$$^{*}\N\cap G_{\V}$ be generators of $\U,\V$ such that

\begin{center} $P(\alpha_{1},...,\alpha_{n})=Q(\beta_{1},...,\beta_{m})=0$. \end{center}

By transfer we also have that $Q($$^{*}\beta_{1},...,$$^{*}\beta_{m})=0$. Now let us consider the elements $\eta_{1},...,\eta_{n},\xi_{1},...,\xi_{m}$ in $G_{\U\odot\V}$ where, for $i\leq n$ and $j\leq m$, we set:

\begin{center} $\eta_{i}=\alpha_{i}\cdot$$^{*}\beta_{1}$ and $\xi_{j}=\alpha_{1}\cdot$$^{*}\beta_{j}$. \end{center}

We know by hypothesis that $P(x_{1},...,x_{n})$ and $Q(y_{1},...,y_{m})$ are homogeneous. So, if $d_{p}$ and $d_{q}$ are their respective degrees, we have:

\begin{center} $P(\eta_{1},...,\eta_{n})+Q(\xi_{1},...,\xi_{m})=$$^{*}\beta_{1}^{d_{p}}P(\alpha_{1},...,\alpha_{n})+\alpha_{1}^{d_{1}}Q($$^{*}\beta_{1},...,$$^{*}\beta_{m})=0+0=0$.\end{center}

But $\eta_{1},...,\eta_{n},\xi_{1},...,\xi_{m}$ are generators of $\U\odot\V$, so we can apply Theorem \ref{PBT} and we obtain the thesis.\end{proof}

We note that if $P$ and $Q$ have the same degree then $P+Q$ is homogeneous, so $P+Q\in\mathcal{H}$.\\
Some results related to Theorem \ref{summ} are proved in \cite{Tesi}.\\
An interesting property of the set $\mathcal{H}$ is that it is closed under the following particular operation:

\begin{center} $P(x_{1},...,x_{n})\rightarrow (x_{1}\cdot...\cdot x_{n})^{d_{P}}P(\frac{1}{x_{1}},...,\frac{1}{x_{n}})$, \end{center}

where $d_{P}$ is the degree of $P(x_{1},...,x_{n})$. This property has been proved by T.C. Brown and V. Rödl in \cite{rif4} in a more general formulation. Here we present a nonstandard proof:

\begin{thm} Let $P(x_{1},...x_{n})$ be a polynomial in $\mathcal{H}$ and let $d$ be the degree of $P(x_{1},...,x_{n})$. Then the polynomial $Q(x_{1},...,x_{n})=x_{1}^{d}\cdot....\cdot x_{n}^{d}\cdot P(\frac{1}{x_{1}},...,\frac{1}{x_{n}})$ is in $\mathcal{H}$. \end{thm}

\begin{proof} Let $\U$ be a $\sigma_{P}$-ultrafilter and let $\alpha_{1},...,\alpha_{n}\in$$^{*}\N$ be generators of $\U$ with $P(\alpha_{1},...,\alpha_{n})=0$. Let us consider the hyperrational numbers 

\begin{center} $\beta_{1}=\frac{1}{\alpha_{1}}$, ... ,$\beta_{n}=\frac{1}{\alpha_{n}}$. \end{center}

By construction, $\beta_{1},...,\beta_{n}$ are generators of the ultrafilter $\frac{1}{\U}\in\beta\mathbb{Q}$. Moreover, 

\begin{center} $P(\frac{1}{\beta_{1}},...,\frac{1}{\beta_{n}})=0$,\end{center}

so $Q(\beta_{1},...,\beta_{n})=0$. But $P(x_{1},...,x_{n})$ is homogeneous so $Q(\beta_{1},...,\beta_{n})$ is also homogeneous and, for every $\gamma\in$$^{\bullet}\N$, we have that 

\begin{center} $Q(\gamma\cdot \beta_{1},...,\gamma\cdot\beta_{n})=0$. \end{center}

Now we consider the hypernatural number 

\begin{center} $\eta=(\alpha_{1})!$ \end{center}

(where $(\alpha_{1})!$ denotes the factorial of $\alpha_{1}$), and we set

\begin{center}$\gamma=$$^{*}\eta$.\end{center}

Then $\gamma$ satisfies the following two properties: 
\begin{enumerate}

	\item for every $i\leq n$ we have that $\gamma\cdot\beta_{i}\in G_{\frac{1}{\U}\odot\V}$, where $\V=\mathfrak{U}_{\gamma}$;
	\item $\gamma\cdot\beta_{i}\in$$^{\bullet}\N$ for every $i\leq n$.

\end{enumerate}

In fact:

\begin{enumerate}

	\item this property follows directly from Proposition \ref{stanfa} (or, to be precise, from the analogous property that holds for $\beta\mathbb{Q}$, see \cite{Tesi});
	\item let us observe that $\gamma=$$^{*}(\alpha_{1}!)=$$(^{*}\alpha_{1})!$. So, since $^{*}\alpha_{1}>\alpha_{i}$ for every $i\leq n$ (because $^{*}\alpha_{1}\in$$^{**}\N\setminus$$^{*}\N$, while $\alpha_{i}\in$$^{*}\N$), $\gamma$ is divisible by each $\alpha_{i}$, which means that $\gamma\cdot\beta_{i}=\frac{\gamma}{\alpha_{i}}\in$$^{\bullet}\N$ for every $i\leq n$.
	\end{enumerate}
	
So by (1) it follows that, for every $i\leq n$, $\gamma\cdot\beta_{i}$ are generators of the ultrafilter $\frac{1}{\U}\odot\V\in\beta\mathbb{Q}$; by (2) it follows that, in particular, $\frac{1}{\U}\odot\V\in\beta\mathbb{N}$. But we know that $Q(\gamma\cdot\beta_{1},...,\gamma\cdot\beta_{n})=0$, so we can apply the Polynomial Bridge Theorem and we can conclude.\end{proof}

We conclude this section by showing an example of a result on the partition regularity of nonlinear polynomials that has been proved in \cite{Mono}, that we recall:

\begin{defn} Let $m$ be a positive natural number, and let $\{y_{1},...,y_{m}\}$ be a set of mutually distinct variables. For every finite set $F\subseteq\{1,..,m\}$ we denote by $Q_{F}(y_{1},...,y_{m})$ the monomial
\begin{center} $Q_{F}(y_{1},...,y_{m})=\begin{cases} \prod\limits_{j\in F} y_{j}, & \mbox{if  } F\neq \emptyset;\\ 1, & \mbox{if  } F=\emptyset.\end{cases}$ \end{center}
\end{defn}
\begin{thm}\label{lev} Let $n\geq 2$ be a natural number, let $R(x_{1},...,x_{n})=\sum\limits_{i=1}^{n} a_{i}x_{i}$ be a partition regular polynomial, and let $m$ be a positive natural number. Then, for every $F_{1},...,F_{n}\subseteq\{1,..,m\}$ $($with the request that, when $n=2$, $F_{1}\cup F_{2}\neq\emptyset)$, the polynomial
\begin{center} $P(x_{1},...,x_{n},y_{1},...,y_{m})=\sum\limits_{i=1}^{n} a_{i}x_{i}Q_{F_{i}}(y_{1},...,y_{m})$ \end{center}
is partition regular. \end{thm}

E.g., by Theorem \ref{lev} it follows that the polynomial

\begin{center} $P(x_{1},x_{2},x_{3},y_{1},y_{2})=2x_{1}y_{1}+7x_{2}-2x_{3}y_{1}y_{2}$ \end{center}

is partition regular.\\
A proof and a generalization of Theorem \ref{lev} can be found in \cite{Mono}: here we want to give a direct proof of a very special case of Theorem \ref{lev}, namely that the simple nonlinear polynomial $x+y-zw$ is partition regular (we point out that the question on the partition regularity of this polynomial was asked by the authors in \cite{rif6}, and was solved by Neil Hindman in \cite{rif13}).\\
To prove that the polynomial $x+y-zw$ is partition regular we use the following lemma:

\begin{lem}\label{fuffi} If $P(x_{1},...,x_{n})$ is an homogeneous partition regular polynomial then there exists a multiplicative idempotent $\sigma_{P}-$ultrafilter. \end{lem}

\begin{proof} This lemma is proved in \cite{Mono}. \end{proof}

\begin{prop} The polynomial $P(x,y,z,w)=x+y-zw$ is partition regular. \end{prop}

\begin{proof} Let us consider the linear polynomial $Q(x,y,z)=x+y-z$. Let $\U$ be a multiplicatively idempotent $\sigma_{Q}$-ultrafilter. Let $\alpha,\beta,\gamma\in G_{\U}\cap$$^{*}\N$ be such that $\alpha+\beta-\gamma=0$. Then $\alpha\cdot$$^{*}\alpha+\beta\cdot$$^{*}\alpha-\gamma\cdot$$^{*}\alpha=0$. If we pose $\xi_{1}=\alpha\cdot$$^{*}\alpha$, $\xi_{2}=\beta\cdot$$^{*}\alpha$, $\xi_{3}=\alpha$ and $\xi_{4}=$$^{*}\alpha$ we have that $\xi_{1},\xi_{2},\xi_{3}$ and $\xi_{4}$ are in $G_{\U}$ and $P(\xi_{1},\xi_{2},\xi_{3}\xi_{4})=0$. So we can apply the Polynomial Bridge Theorem and we can conclude.\end{proof}

\section{Conclusions}

We conclude the paper by pointing out two unsolved questions that arise from the examples that we gave in section 3.3. The first question that we find really interesting is the following: is there a characterization of $\mathcal{P}$ that, given a polynomial $P(x_{1},...,x_{n})$, allows to decide in a finite time if $P(x_{1},...,x_{n})\in\mathcal{P}$ or not?\\
Since we think that the previous question should be really challenging, we pose one other question, that we hope to be simpler to answer: is there a characterization of $\mathcal{H}$ that, given a polynomial $P(x_{1},...,x_{n})$, allows to decide in a finite time if $P(x_{1},...,x_{n})\in\mathcal{H}$ or not?\\
We think that this new question should be easier to answer thanks to Lemma \ref{fuffi}, although it will still probably be really difficult to find such a characterization.

\end{document}